\documentclass{article}
\usepackage{amssymb}
\usepackage{latexsym}
\usepackage{amsmath}
\usepackage{mathrsfs}
\usepackage{color}
\usepackage{CJK}
\usepackage{graphicx}

\newtheorem{theorem}{Theorem}[section]

\newtheorem{lemma}{Lemma}[section]
\newtheorem{proposition}{Proposition}[section]
\newtheorem{remark}{Remark}[section]

\numberwithin{equation}{section}
\newenvironment{proof}{\medskip\par\noindent{\bf Proof.}\ }{\qquad
\raisebox{-0.5mm}{\rule{1.5mm}{4mm}}\vspace{6pt}}

\newcommand{\bbr}{\mathbb{R}}

\newcommand{\h}{H^1_0(\mathbb{B}_R)}

\newcommand{\bbn}{\mathbb{N}}

\begin{document}
\title
{\Large\bf On finding solutions of a Kirchhoff type problem}%

\author{
Yisheng Huang$^{a},$\thanks{E-mail address: yishengh@suda.edu.cn(Yisheng Huang)}\quad
Zeng Liu$^{b},$\thanks{E-mail address: luckliuz@163.com(Zeng Liu)}\quad
Yuanze Wu$^{c}$\thanks{Corresponding
author. E-mail address: wuyz850306@cumt.edu.cn (Yuanze Wu).}\\%
\footnotesize$^{a}${\em  Department of Mathematics, Soochow University, Suzhou 215006, P.R. China }\\%
\footnotesize$^{b}${\em  Department of Mathematics, Suzhou University of Science and Technology, Suzhou 215009, P.R. China}\\
\footnotesize$^{c}${\em  College of Sciences, China University of Mining and Technology, Xuzhou 221116, P.R. China }}%
\date{}
\maketitle

\noindent{\bf Abstract:} Consider the following Kirchhoff type problem
$$
\left\{\aligned -\bigg(a+b\int_{\mathbb{B}_R}|\nabla u|^2dx\bigg)\Delta u&= \lambda u^{q-1} + \mu u^{p-1}, &\quad \text{in }\mathbb{B}_R, \\
u&>0,&\quad\text{in }\mathbb{B}_R,\\
u&=0,&\quad\text{on }\partial\mathbb{B}_R,
\endaligned
\right.\eqno{(\mathcal{P})}
$$
where $\mathbb{B}_R\subset \bbr^N(N\geq3)$ is a ball, $2\leq q<p\leq2^*:=\frac{2N}{N-2}$ and $a$, $b$, $\lambda$, $\mu$ are positive parameters.  By introducing some new ideas and using the well-known results of the problem $(\mathcal{P})$ in the cases of $a=\mu=1$ and $b=0$, we obtain some special kinds of solutions to $(\mathcal{P})$ for all $N\geq3$ with precise expressions on the parameters $a$, $b$, $\lambda$, $\mu$, which  reveals some new phenomenons of the solutions to the  problem $(\mathcal{P})$.  It is also worth to point out that it seems to be the first time that the solutions of $(\mathcal{P})$ can be expressed precisely on the parameters $a$, $b$, $\lambda$, $\mu$, and our results in dimension four also give a partial answer to Neimen's open problems [J. Differential Equations, 257 (2014), 1168--1193].  Furthermore, our results in dimension four seems to be almost ``optimal''.%

\vspace{6mm} \noindent{\bf Keywords:} Kirchhoff type problem; critical Sobolev exponent; positive solution.%

\vspace{6mm}\noindent {\bf AMS} Subject Classification 2010: 35B09; 35B33; 35J15; 35J60.%

\section{Introduction}
In this paper, we study the following Kirchhoff type problem
$$
\left\{\aligned -\bigg(a+b\int_{\Omega}|\nabla u|^2dx\bigg)\Delta u&= \lambda u^{q-1} + \mu u^{p-1}, &\quad \text{in }\Omega, \\
u&>0,&\quad\text{in }\Omega,\\
u&=0,&\quad\text{on }\partial\Omega,
\endaligned
\right.\eqno{(\mathcal{P}_{a,b,\lambda,\mu})}
$$
where $\Omega\subset \bbr^N(N\geq3)$ is a bounded domain with smooth boundary, $2\leq q<p\leq2^*:=\frac{2N}{N-2}$, $2^*$ is the critical Sobolev exponent and $a$, $b$, $\lambda$, $\mu$ are positive parameters.%

The elliptic type Kirchhoff problem (Kirchhoff type problem for short) in a domain $\Omega\subset \bbr^N(1\leq N\leq3)$ has a nice background in physics.  Indeed, such problem is related to the stationary analogue of the following model:
\begin{equation}\label{eq001}
\left\{\aligned &u_{tt}-\bigg(a+b\int_{\Omega}|\nabla u|^2dx\bigg)\Delta u=h(x,u)\quad\text{in }\Omega\times(0, T),\\
&u=0\quad\text{on }\partial\Omega\times(0, T),\\
&u(x,0)=u_0(x),\quad u_t(x,0)=u^*(x),\endaligned\right.
\end{equation}
where $T>0$ is a constant, $u_0, u^*$ are continuous functions.  Such model was first proposed by Kirchhoff in 1883 as an extension of the classical D'Alembert's wave equations for free vibration of elastic strings, Kirchhoff's model takes into account the changes in length of the string produced by transverse vibrations.  In \eqref{eq001}, $u$ denotes the displacement, the nonlinearity $h(x,u)$ denotes the external force and the parameter $a$ denotes the initial tension while the parameter $b$ is related to the intrinsic properties of the string (such as Young¡¯s modulus).  For more details on the physical background of Kirchhoff type problems, we refer the readers to \cite{A12,K83}.

Under some suitable assumptions on the nonlinearities, the Kirchhoff type problem has a variational structure in some proper Hilbert spaces.  Thus, it is natural to study the Kirchhoff type problem by the variational method.  However, since the Kirchhoff term $-b(\int_{\Omega}|\nabla u|^2dx)\Delta u$ is non-local and $u\mapsto -b(\int_{\Omega}|\nabla u|^2dx)\Delta u$ is not weakly continuous, a typical difficulty of such problem by using the variational method is that the weak limit of the $(PS)$ sequence to the corresponding functional is not trivially to be the weak solution of the equation.  In order to overcome this difficulty, several methods have been developed (cf. \cite{CWL12,HZ12,LLS12,PZ06,ZP06}).  Based on these ideas, various existence and multiplicity results of nontrivial solutions for the Kirchhoff type problem in a domain $\Omega\subset \bbr^N(1\leq N\leq3)$ have been established by the variational method in the literatures, see for example \cite{ACM05,F13,LLS14,LLT15,N14,W15} and the references therein for the bounded domain and \cite{AF12,HLP14,LY131,SW14,WTXZ12} and the references therein for the whole space.

Recently, the Kirchhoff type problem in high dimensions $(N\geq4)$ has begun to attract much attention.  From the view point of calculus of variation, such problem is much more complex and difficult since the order of the Kirchhoff type non-local term $-b(\int_{\Omega}|\nabla u|^2dx)\Delta u$ in the corresponding functional is $4$, which equals the critical Sobolev exponent $2^*$ in the case of $N=4$ and is greater than $2^*$ in the case of $N\geq5$.  This fact leads to a big difficulty to obtain the boundedness of the $(PS)$ sequence for the corresponding functional.  By making some very careful and complex analyses on the (PS) sequence, several existence and multiplicity results of nontrivial solutions have been established by the variational method in the literatures, see for example \cite{ACM05,CKW11,HLW15,N141,ZP06} and the references therein.  Based on the above facts, it is natural to ask: {\it Can we find some other simple methods to deal with the Kirchhoff type problem?}  In this paper, we will introduce some new ideas to treat the Kirchhoff type problem $(\mathcal{P}_{a,b,\lambda,\mu})$, which seem to be more simple than that of variation.  Our results will also give the precise expressions of the solutions to $(\mathcal{P}_{a,b,\lambda,\mu})$ on the parameters $a$, $b$, $\lambda$, $\mu$ and reveal some new phenomenons of the solutions to $(\mathcal{P}_{a,b,\lambda,\mu})$.

Our main idea is to establish a relation between solutions of $(\mathcal{P}_{a,b,\lambda,\mu})$ and the following equation $(\mathcal{P}_{\alpha})$ by means of a scaling technique:
$$
\left\{\aligned -\Delta u&= \alpha u^{q-1}+u^{p-1}, &\quad \text{in }\Omega, \\
u&>0,&\quad\text{in }\Omega,\\
u&=0,&\quad\text{on }\partial\Omega.
\endaligned
\right.\eqno{(\mathcal{P}_{\alpha})}
$$
This relation can be stated as follows and its proof will be given in Section~2.
\begin{proposition}\label{prop0001}
Let $u_\alpha$ be a solution of $(\mathcal{P}_{\alpha})$ and let
\begin{eqnarray*}
f_{a,b,\lambda,\mu}(\alpha):=a\bigg(\frac{\alpha\mu^{\frac{q-2}{p-2}}}{\lambda}\bigg)^{\frac{p-2}{p-q}}
+b\bigg(\frac{\alpha\mu^{\frac{q-2}{p-2}}}{\lambda}\bigg)^{\frac{p-4}{p-q}}\mu^{\frac{2}{2-p}}\int_{\Omega}|\nabla u_{\alpha}|^2dx.
\end{eqnarray*}
Then $(\mathcal{P}_{a,b,\lambda,\mu})$ has a solution $\bigg(\frac{\lambda}{\alpha\mu}\bigg)^{\frac{1}{p-q}}u_{\alpha}$ if and only if $f_{a,b,\lambda,\mu}(\alpha)=1$.
\end{proposition}

\begin{remark}{\em
\begin{enumerate}
\item[$(a)$] By Proposition~\ref{prop0001}, we can obtain some special kinds of solutions to $(\mathcal{P}_{a,b,\lambda,\mu})$ with precise expressions on the parameters $a$, $b$, $\lambda$, $\mu$ by solving the equation $f_{a,b,\lambda,\mu}(\alpha)=1$ for $\alpha$.  Furthermore, unlike the the variational method, our method does not  need to analyse the (PS) sequence of the corresponding functional to $(\mathcal{P}_{a,b,\lambda,\mu})$.
\item[$(b)$] The proof of Proposition~\ref{prop0001} is based upon an idea used in \cite{ACM05} and  \cite{WHL15} to respectively  treat two kinds of Kirchhoff type problems:
\begin{equation}\label{eq9999}
\left\{\aligned -\bigg(a+b\int_{\Omega}|\nabla u|^2dx\bigg)\Delta u&=f(x), &\quad \text{in }\Omega, \\
u&=0,&\quad\text{on }\partial\Omega
\endaligned
\right.
\end{equation}
and
\begin{equation}\label{eq9998}
\left\{\aligned&\bigg(\alpha\int_{\bbr^N}(|\nabla u|^2+u^2)dx+\beta\bigg)(-\Delta u+u)=|u|^{p-2}u&\text{ in }\bbr^N,\\
&u\in H^1(\bbr^N),\endaligned\right.
\end{equation}
where $2<p<2^*$. Since the local terms of \eqref{eq9999} and \eqref{eq9998} are homogeneous, by using a scaling techniques, one can obtain solutions of \eqref{eq9999} and \eqref{eq9998} by solving two equations whose properties are very clear (cf. \cite{ACM05,WHL15}).
Unlike \eqref{eq9999} and \eqref{eq9998}, the local term of $(\mathcal{P}_{a,b,\lambda,\mu})$ is inhomogeneous, we need to solve a more difficult equation $f_{a,b,\lambda,\mu}(\alpha)=1$ to obtain solutions of $(\mathcal{P}_{a,b,\lambda,\mu})$ due to the function $\int_{\Omega}|\nabla u_{\alpha}|^2dx$.
\item[$(c)$] Our method can also be used to deal with the Kirchhoff type problem with the nonlinearities $\sum_{i=1}^n\theta_iu^{p_i}$, where $\theta_i$ are constants and $2\leq p_i\leq 2^*$ for all $i=1,2,\cdots,n$.  In this case, in order to observe a similar result to Proposition~\ref{prop0001}, a more complex $n$--components nonlinear system need to be studied.
\end{enumerate}}
\end{remark}

According to Proposition~\ref{prop0001}, in order to obtain solutions of $(\mathcal{P}_{a,b,\lambda,\mu})$, we need to solve the equation $f_{a,b,\lambda,\mu}(\alpha)=1$ in $\bbr$. However,  due to the function $\int_{\Omega}|\nabla u_{\alpha}|^2dx$, this equation is not easy to solve on a general bounded domain $\Omega$ (more reasons will be given in $(a)$ of Remark~\ref{rmk0003}).  For the sake of demonstrating well our ideas, we  mainly consider the problem $(\mathcal{P}_{a,b,\lambda,\mu})$ with $\Omega=\mathbb{B}_R$, i.e.
$$
\left\{\aligned -\bigg(a+b\int_{\mathbb{B}_R}|\nabla u|^2dx\bigg)\Delta u&= \lambda u^{q-1} + \mu u^{p-1}, &\quad \text{in }\mathbb{B}_R, \\
u&>0,&\quad\text{in }\mathbb{B}_R,\\
u&=0,&\quad\text{on }\partial\mathbb{B}_R,
\endaligned
\right.
$$
where $\mathbb{B}_R\subset \bbr^N(N\geq3)$ is a ball.

Before we state results,  we shall give some notations.  Let $\mathcal{E}(u):=\frac12\|\nabla u\|_{L^2(\mathbb{B}_R)}^2-\frac1p\|u\|_{L^p(\mathbb{B}_R)}^p$, where $\|\cdot\|_{L^r(\mathbb{B}_R)}(r\geq1)$ is the usual norm in $L^r(\mathbb{B}_R)$.  Then it is easy to see that $\mathcal{E}(u)$ is of $C^2$ in $\h$.  Furthermore, positive critical points of $\mathcal{E}(u)$ are solutions of $(\mathcal{P}_0)$.  Let
\begin{eqnarray*}
\mathcal{N}:=\{u\in\h\mid\mathcal{E}'(u)u=0\}
\end{eqnarray*}
and define $m_0:=\inf_{u\in\mathcal{N}}\mathcal{E}(u)$.  Now, our main results in this paper can be stated as follows.
\begin{theorem}\label{thm0001}
Let $a,b,\lambda,\mu>0$, $\Omega=\mathbb{B}_R$ and $2=q<p<2^*$.
\begin{enumerate}
\item[$(1)$] $(\mathcal{P}_{a,b,\lambda,\mu})$ has a radial solution if one of the following four cases holds:
\begin{enumerate}
\item[$(i)$] $p>4$ and $\lambda<a\lambda_1$;
\item[$(ii)$] $p=4$, $\lambda<a\lambda_1$ and $\frac{2p}{p-2}m_0 b\mu^{-1}<1$;
\item[$(iii)$] $p=4$, $\lambda>a\lambda_1$ and $\frac{2p}{p-2}m_0 b\mu^{-1}>1$;
\item[$(iii)$] $p<4$ and $\lambda>a\lambda_1$.
\end{enumerate}
\item[$(2)$] If
\begin{eqnarray*}
\frac{2}{(p-2)\mu}\bigg(\frac{(p-2)a}{4-p}\bigg)^{\frac{4-p}{2}}
\bigg(b\lambda_1|\mathbb{B}_R|^{\frac{p-2}{p}}\bigg(\frac{2p}{p-2}m_0\bigg)^{\frac2p}+b\frac{2p}{p-2}m_0\bigg)^{\frac{p-2}{2}}<1
\end{eqnarray*}
then $(\mathcal{P}_{a,b,\lambda,\mu})$ has two radial solutions in the case of $p<4$ and $\lambda<a\lambda_1$.
\end{enumerate}
\end{theorem}

\begin{theorem}\label{thm0002}
Let $a,b,\lambda,\mu>0$, $\Omega=\mathbb{B}_R$ and $2<q<p<2^*$.
\begin{enumerate}
\item[$(1)$] $(\mathcal{P}_{a,b,\lambda,\mu})$ has a radial solution if one of the following two cases holds:
\begin{enumerate}
\item[$(i)$] $N=3$, $p>4$ and $(q-1)(p+1)\leq\frac 32$;
\item[$(ii)$] $N=3$, $p=4$, $(q-1)(p+1)\leq\frac 32$ and $\frac{2p}{p-2}m_0 b\mu^{-1}<1$.
\end{enumerate}
\item[$(2)$] If $p<4$ and
\begin{eqnarray*}
\frac{2}{(p-2)\mu}\bigg(\frac{(p-2)a}{(4-p)}\bigg)^{\frac{4-p}{2}}\bigg(\frac{2q m_0b}{q-2}\bigg)^{\frac{p-2}{2}}<1
\end{eqnarray*}
then $(\mathcal{P}_{a,b,\lambda,\mu})$ has two radial solutions under one of the following two cases:
\begin{enumerate}
\item[$(i)$] $3\leq N\leq 5$ and $(q-1)(p+1)\leq\frac N2$;
\item[$(ii)$] $N\geq6$.
\end{enumerate}
\end{enumerate}
\end{theorem}

\begin{theorem}\label{thm0003}
Let $a,b,\lambda,\mu>0$, $\Omega=\mathbb{B}_R$ and $2=q<p=2^*$.
\begin{enumerate}
\item[$(1)$] $(\mathcal{P}_{a,b,\lambda,\mu})$ has a radial solution if one of the following five cases holds:
\begin{enumerate}
\item[$(i)$] $0<\lambda<a\lambda_1$ and $\frac{a}{4}+\frac{b\mathcal{S}^{\frac32}}{2\mu^{\frac12}}>1$ in the case $N=3$;
\item[$(ii)$] $\lambda>a\lambda_1$ and $\frac{a}{4}+\frac{b\mathcal{S}^{\frac32}}{2\mu^{\frac12}}<1$ in the case $N=3$;
\item[$(iii)$] $0<\lambda<a\lambda_1$ and $\mu>b\mathcal{S}^2$ in the case $N=4$;
\item[$(iv)$] $\lambda>a\lambda_1$ and $\mu<b\mathcal{S}^2$ in the case $N=4$;
\item[$(v)$] $\lambda>a\lambda_1$ in the case $N\geq5$,
\end{enumerate}
where $\mathcal{S}>0$ is the usual Sobolev constant given by%
\begin{equation*}
\mathcal{S}=\inf\{\|\nabla u\|_{L^2(\Omega)}^2\mid u\in H_0^{1}(\Omega), \|u\|_{L^{2^*}(\Omega)}^2=1\}.%
\end{equation*}

\item[$(2)$] If
\begin{eqnarray*}
\frac{N-2}{2\mu}\bigg(\frac{2a}{N-4}\bigg)^{\frac{N-4}{N-2}}
\bigg(b\lambda_1|\mathbb{B}_R|^{\frac2N}\mathcal{S}^{\frac{N-2}{2}}+b\mathcal{S}^{\frac N2}\bigg)^{\frac{2}{N-2}}<1
\end{eqnarray*}
then $(\mathcal{P}_{a,b,\lambda,\mu})$ has two radial solutions in the case $0<\lambda<a\lambda_1$ and $N\geq5$.
\end{enumerate}
\end{theorem}

\begin{theorem}\label{thm0004}
Let $a,b,\lambda,\mu>0$, $\Omega=\mathbb{B}_R$ and $2<q<p=2^*$.
\begin{enumerate}
\item[$(1)$] $(\mathcal{P}_{a,b,\lambda,\mu})$ has a radial solution if one of the following two cases holds:
\begin{enumerate}
\item[$(i)$] N=3 and $a\bigg(\frac{\lambda_0}{\lambda}\bigg)^{\frac{4}{6-q}}\mu^{\frac{q-2}{6-q}}+bC
\bigg(\frac{\lambda_0}{\lambda}\bigg)^{\frac{2}{6-q}}\mu^{\frac{q-4}{6-q}}<1$, where $C=\frac{2q}{N(q-2)}\mathcal{S}^{\frac N2}$ and $\lambda_0>0$ is a constant given in \cite{CZ12};
\item[$(ii)$] $N=4$ and $\mu>b\mathcal{S}^2$.
\end{enumerate}
\item[$(2)$] If\begin{eqnarray*}
\frac{N-2}{2\mu}\bigg(\frac{2a}{N-4}\bigg)^{\frac{N-4}{N-2}}\bigg(\frac{2q b}{N(q-2)}\bigg)^{\frac{2}{N-2}}\mathcal{S}^{\frac {N}{N-2}}<1
\end{eqnarray*}
then $(\mathcal{P}_{a,b,\lambda,\mu})$ has two radial solutions in the case $N\geq5$.
\end{enumerate}
\end{theorem}

\begin{remark}{\em
\begin{enumerate}
\item[$(a)$] Some existence results of Theorems~\ref{thm0001}--\ref{thm0004} in the cases $N=3,4$ have been obtained in the literatures, see for example \cite{F13,HLW15,N14,N141} and the references therein.  Comparing with these papers, the novelty of Theorems~\ref{thm0001}--\ref{thm0004} in the cases $N=3,4$ is that, we can precisely give the range of the parameters $a,b,\lambda,\mu$ and the solutions founded in Theorems~\ref{thm0001}--\ref{thm0004} have precise expressions on the parameters $a,b,\lambda,\mu$ due to Proposition~\ref{prop0001}.
\item[$(b)$] A new and interesting phenomenon revealed by Theorems~\ref{thm0001} and \ref{thm0003} is that the Kirchhoff type problem~$(\mathcal{P}_{a,b,\lambda,\mu})$ with $q=2$ still has solutions if $\lambda>a\lambda_1$ and some further conditions hold, which is quite different from the related local problem~$(\mathcal{P}_{1,0,\lambda,1})$ with $q=2$, for example the well known Brez\'is--Nirenberg problem.
\item[$(c)$] In \cite{N141}, Neimen obtained the following results by using the variational method:
\vspace{6pt}

{\bf Theorem A}{\em \quad
Let $N=4$ and $2<q<4$.  If $b\mathcal{S}^2<\mu<2b\mathcal{S}^2$ and $\Omega\subset \bbr^4$ is strictly star-sharped, then Problem~$(\mathcal{P}_{a,b,\lambda,\mu})$ has a solution under one of the following three cases:
\begin{itemize}
  \item [$(C1)$] $a>0$, $\lambda>0$ is small enough,
  \item [$(C2)$] $\lambda>0$, $a>0$ is large enough,
  \item [$(C3)$] $a>0$, $\lambda>0$ and $\frac{\mu}{b}>\mathcal{S}^2$ is
  sufficiently close to $\mathcal{S}^2$.
\end{itemize}}

Neimen also asked  whether that the conditions that $\mu<2b\mathcal{S}^2$, $\Omega\subset \bbr^4$ is strictly star-sharped and $(C1)$--$(C3)$ are necessary in Theorem~A.  In our paper \cite{HLW15}, we give a partial answer to Neimen's open question, where, by using the variational method, it has been proved that the conditions that $\mu<2b\mathcal{S}^2$ and $\Omega\subset \bbr^4$ is strictly star-sharped are not necessary in Theorem~A if the parameter $b>0$ is sufficiently small.  Now, by Theorem~\ref{thm0004}, we can give another partial answer to Neimen's open question, that is, in the case $\Omega=\mathbb{B}_R$ the conditions $\mu<2b\mathcal{S}^2$ and $(C1)$--$(C3)$ are not necessary in Theorem~A.
\item[$(d)$] To the best of our knowledge, Theorems~\ref{thm0001}--\ref{thm0004} in the case $N\geq5$ are totally new.
\end{enumerate}}
\end{remark}

\begin{remark}\label{rmk0003}{\em
\begin{enumerate}
\item[$(a)$] The proofs of Theorems~\ref{thm0001}--\ref{thm0004} depend heavily on the continuity of the function $f_{a,b,\lambda,\mu}(\alpha)$ given in Proposition~\ref{prop0001} on some intervals of $\bbr$, which is ensured by the assumption $\Omega=\mathbb{B}_R$.  For a general bounded domain $\Omega$, if we can find a continuous curve $\mathcal{L}$ in the set $\mathbb{S}$ on some intervals of $\bbr$, then $f_{a,b,\lambda,\mu}(\alpha)$ is still continuous on these intervals and the proofs of Theorems~\ref{thm0001}--\ref{thm0004} do work, where $\mathbb{S}=\{(u_\alpha,\alpha)\mid u_\alpha\text{ is a solution of }(\mathcal{P}_\alpha)\}$.  It follows that the answer of Neimen's open question may be positive since it can be solved by finding a continuous curve $\mathcal{L}$ in the set $\mathbb{S}$ in the case $N=4$ and $2<q<p=2^*=4$.  However, we can only obtain such continuous curve in $\mathbb{S}$ in the case $N\geq3$ and $2=q<p<2^*$ by the Rabinowitz global bifurcation theorem (see more details in Appendix).
\item[$(b)$] The conditions of Theorems~\ref{thm0003} and \ref{thm0004} in the case $N=4$ seem to be almost ``optimal''.  Indeed, in our paper \cite{HLW15}, we have shown that $a\lambda_1-\lambda\geq0$ and $b\mathcal{S}^2-\mu\geq0$ can not hold simultaneously if $(\mathcal{P}_{a,b,\lambda,\mu})$ has a solution in the case $q=2$ and  $(\mathcal{P}_{a,b,\lambda,\mu})$ has no solution in the case $b\mathcal{S}^2-\mu>0$ if $a$ is sufficiently large or $\lambda$ is sufficiently small in the case $2<q$.  However, we do not know whether the conditions of Theorems~\ref{thm0003} and \ref{thm0004} in the cases $N=3$ and $N\geq5$ are almost ``optimal''.
\item[$(c)$] Theorems~\ref{thm0001}--\ref{thm0004} give no information of $(\mathcal{P}_{a,b,\lambda,\mu})$ for $\lambda=a\lambda_1$ in the case $q=2$ and $b\mathcal{S}^2=\mu$ in all cases.  On the other hand, due to the above $(b)$, $(\mathcal{P}_{a,b,\lambda,\mu})$ has no solution even in a general bounded domain in the case $\lambda=a\lambda_1$, $q=2$ and $b\mathcal{S}^2=\mu$.
\item[$(d)$] Due to the Kirchhoff type nonlocal term $-b(\int_{\Omega}|\nabla u|^2dx)\Delta u$, we can see from Theorems~\ref{thm0001}--\ref{thm0004} that the Kirchhoff type  problem~$(\mathcal{P}_{a,b,\lambda,\mu})$ has two solutions in some cases even $\Omega=\mathbb{B}_R$.  It seems that the branch of solutions to the Kirchhoff type problem~$(\mathcal{P}_{a,b,\lambda,\mu})$ is more complex than the related local problem~$(\mathcal{P}_{1,0,\lambda,\mu})$.  On the other hand, some concentration behaviors of the solutions to $(\mathcal{P}_{a,b,\lambda,\mu})$ can be observed by study the properties of the function $\alpha(a,b,\lambda,\mu)$, where $\alpha(a,b,\lambda,\mu)$ is given by Proposition~\ref{prop0001}.  However, we will not go further in this direction in the current paper.
\end{enumerate}}
\end{remark}

Through this paper, $o_n(1)$ will always denote the quantities tending towards zero as $n\to\infty$.%

\section{Setting of the problem}
In this setion, we first give the proof of Proposition~\ref{prop0001}.\vspace{6pt}

\noindent\textbf{Proof of Proposition~\ref{prop0001}.}\quad
Let $\psi=t u_\alpha$.  Since $u_{\alpha}$ is a solution of $(\mathcal{P}_{\alpha})$, it follows that
\begin{eqnarray*}
-\Delta\psi=t(\alpha u_\alpha^{q-1}+u_\alpha^{p-1})=\alpha t^{2-q}\psi^{q-1}+t^{2-p}\psi^{p-1}.
\end{eqnarray*}
Set $t_\mu=\mu^{\frac{1}{2-p}}$.  Then $\psi_{\alpha,\mu}=\mu^{\frac{1}{2-p}} u_\alpha$ is a solution of the following equation:
$$
\left\{\aligned -\Delta u&= \alpha\mu^{\frac{q-2}{p-2}} u^{q-1}+\mu u^{p-1}, &\quad \text{in }\Omega, \\
u&>0,&\quad\text{in }\Omega,\\
u&=0,&\quad\text{on }\partial\Omega.
\endaligned
\right.\eqno{(\mathcal{P}_{\alpha,\mu})}
$$
Let $\varphi=s \psi_{\alpha,\mu}$, then we have
\begin{eqnarray*}
-\bigg(a+b\int_{\Omega}|\nabla \varphi|^2dx\bigg)\Delta \varphi&=&s\bigg(a+s^2b\int_{\Omega}|\nabla \psi_{\alpha,\mu}|^2dx\bigg)(\alpha\mu^{\frac{q-2}{p-2}} \psi_{\alpha,\mu}^{q-1}+\mu \psi_{\alpha,\mu}^{p-1})\\
&=&\bigg(a+s^2b\int_{\Omega}|\nabla \psi_{\alpha,\mu}|^2dx\bigg)(s^{2-q}\alpha\mu^{\frac{q-2}{p-2}}\varphi^{q-1}+s^{2-p}\mu\varphi^{p-1}).
\end{eqnarray*}
Therefore $\varphi=s \psi_{\alpha,\mu}$ is a solution of $(\mathcal{P}_{a,b,\lambda,\mu})$ if and only if $(s,\alpha)$ satisfies the following system:
\begin{eqnarray*}
\left\{\aligned \bigg(a+s^2b\int_{\Omega}|\nabla \psi_{\alpha,\mu}|^2dx\bigg)s^{2-q}\alpha\mu^{\frac{q-2}{p-2}}&=\lambda,\\
\bigg(a+s^2b\int_{\Omega}|\nabla \psi_{\alpha,\mu}|^2dx\bigg)s^{2-p}&=1,
\endaligned
\right.
\end{eqnarray*}
which is equivalent to that $(s,\alpha)$ satisfies the following system:
$$
\left\{\aligned s^{p-q}\alpha\mu^{\frac{q-2}{p-2}}&=\lambda,\\
a\bigg(\frac{\alpha\mu^{\frac{q-2}{p-2}}}{\lambda}\bigg)^{\frac{p-2}{p-q}}
+\bigg(\frac{\alpha\mu^{\frac{q-2}{p-2}}}{\lambda}\bigg)^{\frac{p-4}{p-q}}b\int_{\Omega}|\nabla \psi_{\alpha,\mu}|^2dx&=1.
\endaligned
\right.
$$
So that $\varphi_{a,b,\lambda,\mu}=\bigg(\frac{\lambda}{\alpha\mu^{\frac{q-2}{p-2}}}\bigg)^{\frac{1}{p-q}}\mu^{\frac{1}{2-p}}u_{\alpha}
=\bigg(\frac{\lambda}{\alpha\mu}\bigg)^{\frac{1}{p-q}}u_{\alpha}$ is a solution of $(\mathcal{P}_{a,b,\lambda,\mu})$ if and only if $f_{a,b,\lambda,\mu}(\alpha)=1$.
\qquad\raisebox{-0.5mm}{%
\rule{1.5mm}{4mm}}\vspace{10pt}

Next we will consider the continuity of $f_{a,b,\lambda,\mu}(\alpha)$ as a function of $\alpha$ on some subset of $\mathbb{R}$.  In order to do this, let us respectively denote the corresponding functional and the Nehari manifold of $(\mathcal{P}_\alpha)$ in $H_0^1(\Omega)$ by $I_\alpha(u)$ and $\mathcal{N}_{\alpha}$, that is,
\begin{eqnarray*}\label{eq1000}
I_\alpha(u):=\frac12\|\nabla u\|_{L^2(\Omega)}^2-\frac{\alpha}{q}\|u\|_{L^q(\Omega)}^q-\frac{1}{p}\|u\|_{L^{p}(\Omega)}^{p}
\end{eqnarray*}
and
\begin{eqnarray*}
\mathcal{N}_{\alpha}:=\{u\in H^1_0(\Omega)\mid I_\alpha'(u)u=0\}.
\end{eqnarray*}
Define
\begin{eqnarray*}
\mathcal{D}:=\{\alpha\mid(\mathcal{P}_\alpha)\text{ has a unique solution }u_\alpha\text{ with }I_\alpha(u_\alpha)=m_\alpha\},
\end{eqnarray*}
where $m_\alpha:=\inf_{u\in\mathcal{N}_{\alpha}}I_{\alpha}(u)$.  Then we have the following.
\begin{lemma}\label{lem0002}
If $p<2^*$  then the function $f_{a,b,\lambda,\mu}(\alpha)$ is continuous on $\mathcal{D}$.
\end{lemma}
\begin{proof}
Let $\alpha_0\in\mathcal{D}$ and $\{\alpha_n\}\subset \mathcal{D}$ satisfying $\alpha_n=\alpha_0+o_n(1)$.  By a similar argument used in the proof of \cite[Lemma~5.1]{HWW14}, we can see that $m_{\alpha_n}=m_{\alpha_0}+o_n(1)$.  It follows that $\{u_{\alpha_n}\}$ is bounded in $H_0^1(\Omega)$.  Without loss of generality, we may assume that $u_{\alpha_n}\rightharpoonup u_{\alpha_0}^*$ weakly in $H_0^1(\Omega)$ for some $u_{\alpha_0}^*\in H_0^1(\Omega)$ as $n\to\infty$.  It is easy to show that $I_{\alpha_0}'(u_{\alpha_0}^*)=0$ in $H^{-1}(\Omega)$, where $H^{-1}(\Omega)$ is the dual space of $H_0^1(\Omega)$.  This, together with the Sobolev embedding theorem and the fact that $I_{\alpha_n}'(u_{\alpha_n})=0$ in $H^{-1}(\Omega)$, implies that $u_{\alpha_n}=u_{\alpha_0}^*+o_n(1)$ strongly in $H_0^1(\Omega)$.  Since $u_{\alpha_n}>0$ in $\Omega$, by the strong maximum principle, we also have $u_{\alpha_0}^*>0$ in $\Omega$.  Thus, $u_{\alpha_0}^*$ is a solution of $(\mathcal{P}_{\alpha_0})$ with $I_{\alpha_0}(u_{\alpha_0}^*)=m_{\alpha_0}$.  Since $\alpha_0\in\mathcal{D}$, we must have $u_{\alpha_0}^*=u_{\alpha_0}$ in $H_0^1(\Omega)$.    It follows from the arbitrariness of $\alpha_0\in\mathcal{D}$ that the function $\alpha\mapsto\int_{\Omega}|\nabla u_{\alpha}|^2dx$ is continuous on $\mathcal{D}$, which deduces that $f_{a,b,\lambda,\mu}(\alpha)$ is continuous on $\mathcal{D}$.
\end{proof}

\begin{lemma}\label{lem0003}
If $p=2^*$ then the function $f_{a,b,\lambda,\mu}(\alpha)$ is continuous on $\mathcal{D}\cap\mathcal{F}$, where $\mathcal{F}=\{\alpha\mid m_\alpha<\frac1N S^{\frac N2}\}$.
\end{lemma}
\begin{proof}
Let $\alpha_0\in\mathcal{D}\cap\mathcal{F}$ and $\{\alpha_n\}\subset\mathcal{D}\cap\mathcal{F}$ satisfying $\alpha_n=\alpha_0+o_n(1)$.  Since $I_{\alpha_n}(u_{\alpha_n})<\frac1N\mathcal{S}^{\frac N2}$ and $u_{\alpha_n}$ is a solution of $(\mathcal{P}_{\alpha_n})$, by using standard arguments, we can show that $\{u_{\alpha_n}\}$ is bounded in $H_0^1(\Omega)$.  Going if necessary to a subsequence, we can assume that $u_{\alpha_n}\rightharpoonup u_{\alpha_0}^*$ weakly in $H_0^1(\Omega)$ for some $u_{\alpha_0}^*\in H_0^1(\Omega)$ as $n\to\infty$.  Clearly, $I_{\alpha_0}'(u_{\alpha_0}^*)=0$ in $H^{-1}(\Omega)$.  The strong maximum principle and the fact that $u_{\alpha_n}>0$ in $\Omega$ ensure that $u_{\alpha_0}^*>0$ in $\Omega$.  On the other hand, a similar argument in the proof of \cite[Lemma~5.1]{HWW14} also gives that $m_{\alpha_n}=m_{\alpha_0}+o_n(1)$.  Note that $\alpha_n=\alpha_0+o_n(1)$ and $m_{\alpha_0}<\frac1N\mathcal{S}^{\frac N2}$, we can use the Br\'ezis-Lieb Lemma and the Sobolev embedding theorem in a standard way to show that $u_{\alpha_n}=u_{\alpha_0}^*+o_n(1)$ strongly in $H_0^1(\Omega)$.  Therefore, $u_{\alpha_0}^*$ is a solution of $(\mathcal{P}_{\alpha_0})$ with $I_{\alpha_0}(u_{\alpha_0}^*)=m_{\alpha_0}$.  Since $\alpha_0\in\mathcal{D}$, we must have $u_{\alpha_0}^*=u_{\alpha_0}$ in $H^1_0(\Omega)$.  Thus, we have proved that the function $\alpha\mapsto\int_{\Omega}|\nabla u_{\alpha}|^2dx$ is continuous on $\mathcal{D}\cap\mathcal{F}$ and so that $f_{a,b,\lambda,\mu}(\alpha)$ is continuous on $\mathcal{D}\cap\mathcal{F}$.
\end{proof}

\section{The existence of solutions}
In this section, with the help of Proposition~\ref{prop0001}, we will give the proofs of our results on the existence of solutions to $(\mathcal{P}_{a,b,\lambda,\mu})$ in the case of $\Omega=\mathbb{B}_R$.

\subsection{The case of $2=q<p<2^*$}
It is well-known that $(0, \lambda_1)\subset\mathcal{D}$ in this case and $u_\alpha$ is radial (cf. \cite{K89}).  In order to apply Proposition~\ref{prop0001}, we need the following lemma.
\begin{lemma}\label{lem0004}
It holds that $\lim_{\alpha\uparrow\lambda_1}\int_{\mathbb{B}_R}|\nabla u_{\alpha}|^2dx=0$ and $\lim_{\alpha\downarrow0}\int_{\mathbb{B}_R}|\nabla u_{\alpha}|^2dx=\frac{2p}{p-2}m_0$.
\end{lemma}
\begin{proof}
We first prove the former.  Suppose that $\alpha_n\uparrow\lambda_1$ as $n\to\infty$, then by a similar argument used in the proof of \cite[Lemma~5.2]{HWW14}, we can see that $m_{\alpha_n}\downarrow m_{\lambda_1}$ as $n\to\infty$.  It follows from a standard argument that $\{u_{\alpha_n}\}$ is bounded in $\h$.  Without loss of generality, we may assume that $u_{\alpha_n}\rightharpoonup u_{\lambda_1}$ weakly in $H_0^1(\mathbb{B}_R)$ for some $u_{\lambda_1}\in H_0^1(\mathbb{B}_R)$ as $n\to\infty$. Similarly as in the proof of Lemma~\ref{lem0002}, we obtain that  $u_{\alpha_n}=u_{\lambda_1}+o_n(1)$ strongly in $H_0^1(\mathbb{B}_R)$ and $u_{\lambda_1}$ is a solution of $(\mathcal{P}_{\lambda_1})$ if $u_{\lambda_1}\neq 0$.  Note that $(\mathcal{P}_{\lambda_1})$ has no solution, so we must have $u_{\lambda_1}=0$, which means that $\lim_{\alpha\uparrow\lambda_1}\int_{\mathbb{B}_R}|\nabla u_{\alpha}|^2dx=0$. To prove the later, let us assume that $\alpha_n\downarrow0$ as $n\to\infty$. Similarly as in the above, we can imply that $u_{\alpha_n}=u_{0}+o_n(1)$ strongly in $H_0^1(\mathbb{B}_R)$, where $u_0$ is the ground state solution of $(\mathcal{P}_0)$, so that $\lim_{\alpha\downarrow0}\int_{\mathbb{B}_R}|\nabla u_{\alpha}|^2dx=\frac{2p}{p-2}m_0$.
\end{proof}\vspace{3pt}

With Lemma~\ref{lem0004} in hands, we can give the proof of Theorem~\ref{thm0001}.\vspace{10pt}

\noindent\textbf{Proof of Theorem~\ref{thm0001}.}\quad
By Lemma~\ref{lem0004}, we have
\begin{eqnarray*}
\lim_{\alpha\downarrow0}f_{a,b,\lambda,\mu}(\alpha)=\left\{\aligned&+\infty, &\text{if } p<4,\\&
\frac{2p}{p-2}m_0 b\mu^{-1},&\text{if }p=4,\\
&0, &\text{if }p>4 \endaligned\right.
\end{eqnarray*}
and
\begin{eqnarray*}
\lim_{\alpha\uparrow\lambda_1}f_{a,b,\lambda,\mu}(\alpha)=\frac{a\lambda_1}{\lambda}.
\end{eqnarray*}
It follows from Lemma~\ref{lem0002} that $f_{a,b,\lambda,\mu}(\alpha)=1$ has a solution $\alpha_0>0$ under one of the following four cases:
\begin{enumerate}
\item[$(i)$] $p>4$ and $\lambda<a\lambda_1$;
\item[$(ii)$] $p=4$, $\lambda<a\lambda_1$ and $\frac{2p}{p-2}m_0 b\mu^{-1}<1$;
\item[$(iii)$] $p=4$, $\lambda>a\lambda_1$ and $\frac{2p}{p-2}m_0 b\mu^{-1}>1$;
\item[$(iii)$] $p<4$ and $\lambda>a\lambda_1$.
\end{enumerate}
Furthermore, a similar argument used in the proof of \cite[Lemma~5.2]{HWW14} shows that $m_{\alpha}<m_0$ for all $\alpha\in(0, \lambda_1)$.  It follows from the H\"older inequality that
\begin{eqnarray*}
\int_{\mathbb{B}_R}|\nabla u_{\alpha}|^2dx\leq\lambda_1|\mathbb{B}_R|^{\frac{p-2}{p}}\bigg(\frac{2p}{p-2}m_0\bigg)^{\frac2p}+\frac{2p}{p-2}m_0\quad\text{for all }\alpha\in(0, \lambda_1),
\end{eqnarray*}
which implies that
\begin{eqnarray*}
f_{a,b,\lambda,\mu}(\alpha)\leq\frac{a\alpha}{\lambda}+b\bigg(\lambda_1|\mathbb{B}_R|^{\frac{p-2}{p}}\bigg(\frac{2p}{p-2}m_0\bigg)^{\frac2p}+\frac{2p}{p-2}m_0\bigg)
\bigg(\frac{\alpha}{\lambda}\bigg)^{\frac{p-4}{p-2}}\mu^{\frac{2}{2-p}}.
\end{eqnarray*}
By a direct calculation, we can see that
\begin{eqnarray*}
f_{a,b,\lambda,\mu}\bigg(\frac{\lambda}{\mu}\bigg(\frac{(4-p)b\mathcal{C}}{a(p-2)}
\bigg)^{\frac{p-2}{2}}\bigg)<1
\end{eqnarray*}
under the following condition
\begin{eqnarray*}
\frac{2}{(p-2)\mu}\bigg(\frac{(p-2)a}{4-p}\bigg)^{\frac{4-p}{2}}
(b\mathcal{C})^{\frac{p-2}{2}}<1,
\end{eqnarray*}
where $\mathcal{C}=\lambda_1|\mathbb{B}_R|^{\frac{p-2}{p}}\bigg(\frac{2p}{p-2}m_0\bigg)^{\frac2p}+\frac{2p}{p-2}m_0$.
Thus, in the case $p<4$ and $\lambda<a\lambda_1$, the equation $f_{a,b,\lambda,\mu}(\alpha)=1$ has two solutions $0<\alpha_1<\alpha_2$.  By Proposition~\ref{prop0001}, we  complete the proof.
\qquad\raisebox{-0.5mm}{%
\rule{1.5mm}{4mm}}\vspace{6pt}

\subsection{The case of $2<q<p<2^*$}
In this case, it is well-known that $(0, +\infty)\subset\mathcal{D}$ and $u_\alpha$ is radial if one of the following two conditions holds:
\begin{enumerate}
\item[$(i)$] $3\leq N\leq 5$ and $(q-1)(p+1)\leq N/2$ (cf. \cite{Z95});
\item[$(ii)$] $N\geq6$ (cf. \cite{ET97}).
\end{enumerate}\vspace{3pt}

\noindent\textbf{Proof of Theorem~\ref{thm0002}.}\quad By a similar argument used in proof of Lemma~\ref{lem0004}, we also have $\lim_{\alpha\downarrow0}\int_{\mathbb{B}_R}|\nabla u_{\alpha}|^2dx=\frac{2p}{p-2}m_0$, which implies
\begin{eqnarray*}
\lim_{\alpha\downarrow0}f_{a,b,\lambda,\mu}(\alpha)=\left\{\aligned&+\infty, &\text{if }p<4;\\&
\frac{2p}{p-2}m_0 b\mu^{-1}, &\text{if }p=4;\\
&0, &\text{if }p>4,\endaligned\right.
\end{eqnarray*}
It is easy to check that $\lim_{\alpha\uparrow+\infty}f_{a,b,\lambda,\mu}(\alpha)=+\infty$, so that we deduce from Lemma~\ref{lem0002} that $f_{a,b,\lambda,\mu}(\alpha)=1$ has a solution $\alpha_0>0$ under one of the following two cases:
\begin{enumerate}
\item[$(i)$] $N=3$, $p>4$ and $(q-1)(p+1)\leq 3/2$;
\item[$(ii)$] $N=3$, $p=4$, $(q-1)(p+1)\leq 3/2$ and $\frac{2p}{p-2}m_0 b\mu^{-1}<1$.
\end{enumerate}
Furthermore, similarly as in the proof of Theorem~\ref{thm0001}, we can see that
\begin{eqnarray*}
f_{a,b,\lambda,\mu}\bigg(\frac{\lambda}{\mu}\bigg(\frac{(4-p)b\mathcal{C}_1}{(p-2)a}\bigg)^{\frac{p-q}{2}}\bigg)<1
\end{eqnarray*}
under the following condition
\begin{eqnarray*}
\frac{2}{(p-2)\mu}\bigg(\frac{(p-2)a}{4-p}\bigg)^{\frac{4-p}{2}}(\mathcal{C}_1b)^{\frac{p-2}{2}}<1,
\end{eqnarray*}
where $\mathcal{C}_1=\frac{2q}{q-2}m_0$.  Thus, $f_{a,b,\lambda,\mu}(\alpha)=1$ has two solutions $0<\alpha_1<\alpha_2$ under one of the following two cases:
\begin{enumerate}
\item[$(i)$] $p<4$, $3\leq N\leq 5$ and $(q-1)(p+1)\leq N/2$;
\item[$(ii)$] $p<4$, $N\geq6$.
\end{enumerate}
Now, the conclusions of Theorem~\ref{thm0002} follow from Proposition~\ref{prop0001}.
\qquad\raisebox{-0.5mm}{%
\rule{1.5mm}{4mm}}\vspace{6pt}

\subsection{The case of $2=q<p=2^*$}
In this case, it is well-known that $(0, \lambda_1)\subset\mathcal{D}\cap\mathcal{F}$ and $u_\alpha$ is radial if $N\geq4$ and $(\frac{1}{4}\lambda_1, \lambda_1)\subset\mathcal{D}\cap\mathcal{F}$ and $u_\alpha$ is radial if $N=3$ (cf. \cite{BN83,Z92}).  Similarly as in the previous subsection, we need to establish the following lemma before proving Theorem~\ref{thm0003}.
\begin{lemma}\label{lem0005}
If $N=3$ then $\lim_{\alpha\downarrow\frac14\lambda_1}\int_{\mathbb{B}_R}|\nabla u_{\alpha}|^2dx=\mathcal{S}^{\frac 32}$; if $N\geq4$ then $\lim_{\alpha\downarrow0}\int_{\mathbb{B}_R}|\nabla u_{\alpha}|^2dx=\mathcal{S}^{\frac N2}$ and it holds that $\lim_{\alpha\uparrow\lambda_1}\int_{\mathbb{B}_R}|\nabla u_{\alpha}|^2dx=0$ for all $N\geq3$.
\end{lemma}
\begin{proof}
We first show that $\lim_{\alpha\downarrow\frac14\lambda_1}\int_{\mathbb{B}_R}|\nabla u_{\alpha}|^2dx=\mathcal{S}^{\frac 32}$ if $N=3$.  Indeed, by a direct calculation, we have
\begin{eqnarray*}
\frac13(1-\frac{\alpha}{\lambda_1})\mathcal{S}^{\frac32}<I_\alpha(u_\alpha)<\frac13 \mathcal{S}^{\frac32}\quad\text{for}\quad \frac14\lambda_1<\alpha<\lambda_1.
\end{eqnarray*}
Let $\alpha_n\downarrow\frac{\lambda_1}{4}$.  Similarly as in the proof of Lemma~\ref{lem0003}, we can see that $u_{\alpha_n}\rightharpoonup u_0$ weakly in $H_0^1(\mathbb{B}_R)$ for some $u_0\in H_0^1(\mathbb{B}_R)$ as $n\to\infty$ and $I_{\frac14\lambda_1}'(u_0)=0$ in $H^{-1}(\mathbb{B}_R)$.  Since $(\mathcal{P}_{\frac{\lambda_1}{4}})$ has no solution, by the strong maximum principle, we must have that $u_0=0$ in $H_0^1(\mathbb{B}_R)$.  Note that $I_{\alpha_n}(u_{\alpha_n})<\frac13 \mathcal{S}^{\frac32}$ for all $n\in\bbn$, we get from the Sobolev embedding theorem that either
\begin{enumerate}
\item[$(a)$] $\|\nabla u_{\alpha_n}\|_{L^2(\mathbb{B}_R)}^2=o_n(1)$ or
\item[$(b)$] $\|\nabla u_{\alpha_n}\|_{L^2(\mathbb{B}_R)}^2=\mathcal{S}^{\frac32}+o_n(1)$.
\end{enumerate}
Clearly the case $(b)$ must occur since
$\frac13(1-\frac{\alpha_n}{\lambda_1})\mathcal{S}^{\frac32}<I_{\alpha_n}(u_{\alpha_n})$ and $\alpha_n\downarrow\frac{\lambda_1}{4}$.   It follows from the arbitrariness of $\{\alpha_n\}$ that $\lim_{\alpha\downarrow\frac14\lambda_1}\int_{\mathbb{B}_R}|\nabla u_{\alpha}|^2dx=\mathcal{S}^{\frac 32}$ if $N=3$.

Next, we shall  prove that $\lim_{\alpha\downarrow0}\int_{\mathbb{B}_R}|\nabla u_{\alpha}|^2dx=\mathcal{S}^{\frac N2}$ in the case $N\geq4$.  In fact, since $(\mathcal{P}_\alpha)$ has no solution for $\alpha\leq0$ and
\begin{eqnarray*}
\frac1N(1-\frac{\alpha}{\lambda_1})\mathcal{S}^{\frac N2}<I_\alpha(u_\alpha)<\frac1N \mathcal{S}^{\frac N2}\quad\text{for }0<\alpha<\lambda_1\text{ in the case }N\geq4,
\end{eqnarray*}
similar arguments used the above show that that if $N\geq4$ then $\|\nabla u_{\alpha_n}\|_{L^2(\mathbb{B}_R)}^2=\mathcal{S}^{\frac N2}+o_n(1)$ for each sequence $\{\alpha_n\}$ satisfying $\alpha_n\downarrow0$ as $n\to\infty$, so that $\lim_{\alpha\downarrow0}\int_{\mathbb{B}_R}|\nabla u_{\alpha}|^2dx=\mathcal{S}^{\frac N2}$ if $N\geq4$.

Finally, we will prove that $\lim_{\alpha\uparrow\lambda_1}\int_{\mathbb{B}_R}|\nabla u_{\alpha}|^2dx=0$.
Let $\alpha_n\uparrow\lambda_1$.  By using a similar argument in the proof of the first equality above, we reach that $u_{\alpha_n}\rightharpoonup u_0$ weakly in $H_0^1(\mathbb{B}_R)$ for some $u_0\in H_0^1(\mathbb{B}_R)$ as $n\to\infty$.  Since $(\mathcal{P}_{\lambda_1})$ has no solution, it follows from the strong maximum principle that $u_0=0$ in $H_0^1(\mathbb{B}_R)$.  Note that $\alpha_n\uparrow\lambda_1$, we can see from similar arguments used in the proof of \cite[Lemma~5.2]{HWW14} that
\begin{eqnarray*}
I_{\alpha_{n+1}}(u_{\alpha_{n+1}})\leq I_{\alpha_n}(u_{\alpha_n})\quad\text{for all }n\in\bbn.
\end{eqnarray*}
Now, it implies from $I_{\alpha_1}(u_{\alpha_1})<\frac1N \mathcal{S}^{\frac N2}$ and the Sobolev embedding theorem that $\|\nabla u_{\alpha_n}\|_{L^2(\mathbb{B}_R)}^2=o_n(1)$, so that $\lim_{\alpha\uparrow\lambda_1}\int_{\mathbb{B}_R}|\nabla u_{\alpha}|^2dx=0$.
\end{proof}

\vspace{3pt}

\noindent\textbf{Proof of Theorem~\ref{thm0003}.}\quad By Lemma~\ref{lem0005}, we can see that
\begin{eqnarray*}
\lim_{\alpha\uparrow\lambda_1}f_{a,b,\lambda,\mu}(\alpha)=\frac{a\lambda_1}{\lambda}\quad\text{and}\quad
\lim_{\alpha\downarrow\frac{\lambda_1}{4}}f_{a,b,\lambda,\mu}(\alpha)=\frac{a}{4}+\frac{b\mathcal{S}^{\frac32}}{2\mu^{\frac12}}
\end{eqnarray*}
if $N=3$,
\begin{eqnarray*}
\lim_{\alpha\uparrow\lambda_1}f_{a,b,\lambda,\mu}(\alpha)=\frac{a\lambda_1}{\lambda}\quad\text{and}\quad
\lim_{\alpha\downarrow0}f_{a,b,\lambda,\mu}(\alpha)=\frac{b\mathcal{S}^2}{\mu}
\end{eqnarray*}
if $N=4$ and
\begin{eqnarray*}
\lim_{\alpha\downarrow0}f_{a,b,\lambda,\mu}(\alpha)=+\infty\quad\text{and}\quad
\lim_{\alpha\uparrow\lambda_1}f_{a,b,\lambda,\mu}(\alpha)=\frac{a\lambda_1}{\lambda}
\end{eqnarray*}
if $N\geq5$.  It follows from Lemma~\ref{lem0003} that $f_{a,b,\lambda,\mu}(\alpha)=1$ has a solution $\alpha_0>0$ under one of the following five cases:
\begin{enumerate}
\item[$(i)$] $0<\lambda<a\lambda_1$ and $\frac{a}{4}+\frac{b\mathcal{S}^{\frac32}}{2\mu^{\frac12}}<1$ in the case $N=3$;
\item[$(ii)$] $\lambda>a\lambda_1$ and $\frac{a}{4}+\frac{b\mathcal{S}^{\frac32}}{2\mu^{\frac12}}>1$ in the case $N=3$;
\item[$(iii)$] $0<\lambda<a\lambda_1$ and $\mu>b\mathcal{S}^2$ in the case $N=4$;
\item[$(iv)$] $\lambda>a\lambda_1$ and $\mu<b\mathcal{S}^2$ in the case $N=4$;
\item[$(v)$] $\lambda>a\lambda_1$ in the case $N\geq5$.
\end{enumerate}
Now, similarly as in the proof of Theorem~\ref{thm0001}, we can see that
\begin{eqnarray*}
f_{a,b,\lambda,\mu}\bigg(\frac{\lambda}{\mu}\bigg(\frac{(N-4)b\mathcal{C}_2}{2a}
\bigg)^{\frac{2}{N-2}}\bigg)<1
\end{eqnarray*}
under the following condition
\begin{eqnarray*}
\frac{N-2}{2\mu}\bigg(\frac{2a}{N-4}\bigg)^{\frac{N-4}{N-2}}
(b\mathcal{C}_2)^{\frac{2}{N-2}}<1,
\end{eqnarray*}
where $\mathcal{C}_2=\lambda_1|\mathbb{B}_R|^{\frac2N}\mathcal{S}^{\frac{N-2}{2}}+\mathcal{S}^{\frac N2}$.
Thus, in the case $N\geq5$ and $0<\lambda<a\lambda_1$,  $f_{a,b,\lambda,\mu}(\alpha)=1$ has two solutions $0<\alpha_1<\alpha_2$.
Therefore, the conclusions of Theorem~\ref{thm0003} remain true from Proposition~\ref{prop0001}.
\qquad\raisebox{-0.5mm}{%
\rule{1.5mm}{4mm}}\vspace{6pt}

\subsection{The case of $2<q<p=2^*$}
In this case, it is well-known that $(0, +\infty)\subset\mathcal{D}\cap\mathcal{F}$ and $u_\alpha$ is radial if $N\geq4$ (cf. \cite{CZ12,ET97}) and $(\lambda_0, +\infty)\subset\mathcal{D}\cap\mathcal{F}$ and $u_\alpha$ is radial for some $\lambda_0>0$ if $N=3$ (cf. \cite{CZ12}).
\vspace{9pt}

\noindent\textbf{Proof of Theorem~\ref{thm0004}.}\quad By using a similar argument in the proof of Lemma~\ref{lem0005}, we can see that $\lim_{\alpha\downarrow0}\int_{\mathbb{B}_R}|\nabla u_{\alpha}|^2dx=\mathcal{S}^{\frac N2}$ in if $N\geq4$, which implies
\begin{eqnarray*}
\lim_{\alpha\downarrow0}f_{a,b,\lambda,\mu}(\alpha)=\left\{\aligned&+\infty,&\text{if } N\geq5;\\&
\frac{b\mathcal{S}^2}{\mu}, &\text{if }N=4.\endaligned\right.
\end{eqnarray*}
For the case of  $N=3$,  a similar argument in the proof of  Theorem~\ref{thm0001} shows that $\int_{\mathbb{B}_R}|\nabla u_{\alpha}|^2dx\leq\frac{2q}{N(q-2)}\mathcal{S}^{\frac N2}$ for all $\alpha>\lambda_0$.  It follows that
\begin{eqnarray*}
\lim_{\alpha\downarrow\lambda_0}f_{a,b,\lambda,\mu}(\alpha)\leq a\bigg(\frac{\lambda_0}{\lambda}\bigg)^{\frac{4}{6-q}}\mu^{\frac{q-2}{6-q}}+b\mathcal{C}_3
\bigg(\frac{\lambda_0}{\lambda}\bigg)^{\frac{2}{6-q}}\mu^{\frac{q-4}{6-q}},
\end{eqnarray*}
where $\mathcal{C}_3=\frac{2q}{N(q-2)}\mathcal{S}^{\frac N2}$.  Also, we can  easily check that $\lim_{\alpha\uparrow+\infty}f_{a,b,\lambda,\mu}(\alpha)=+\infty$, so that, by Lemma~\ref{lem0003}, we get that $f_{a,b,\lambda,\mu}(\alpha)=1$ has a solution $\alpha_0>0$ under one of the following two cases:
\begin{enumerate}
\item[$(i)$] $N=3$ and $a\bigg(\frac{\lambda_0}{\lambda}\bigg)^{\frac{4}{6-q}}\mu^{\frac{q-2}{6-q}}+b\mathcal{C}_3
\bigg(\frac{\lambda_0}{\lambda}\bigg)^{\frac{2}{6-q}}\mu^{\frac{q-4}{6-q}}<1$;
\item[$(ii)$] $N=4$ and $\mu>b\mathcal{S}^2$.
\end{enumerate}
Furthermore, similarly as in the proof of Theorem~\ref{thm0001}, we can obtain that
\begin{eqnarray*}
f_{a,b,\lambda,\mu}\bigg(\lambda\mu^{\frac{(2-q)(N-2)}{4}}\bigg(\frac{(N-4)b\mathcal{C}_3}{2a}\bigg)^{\frac{(2-q)N+2q}{2(N-2)}}\bigg)<1
\end{eqnarray*}
under the following condition
\begin{eqnarray*}
\frac{N-2}{2\mu}\bigg(\frac{2a}{N-4}\bigg)^{\frac{N-4}{N-2}}(\mathcal{C}_3b)^{\frac{2}{N-2}}<1.
\end{eqnarray*}
Thus, if $N\geq5$, then $f_{a,b,\lambda,\mu}(\alpha)=1$ has two solutions $0<\alpha_1<\alpha_2$.
Now, the conclusions of Theorem~\ref{thm0004} hold from Proposition~\ref{prop0001}.
\qquad\raisebox{-0.5mm}{%
\rule{1.5mm}{4mm}}\vspace{6pt}

\section{Acknowledgements}
Y. Wu thanks Prof. W. Zou for his friendship, encouragement and enlightening discussions.  Y. Wu is supported by the Fundamental Research Funds for the Central Universities (2014QNA67).

\section{Appendix}
In this section, we will find special kinds of solutions to $(\mathcal{P}_{a,b,\lambda,\mu})$ on a general bounded domain $\Omega$ in the case $2=q<p<2^*$.
It is well-known that $(\mathcal{P}_{\alpha})$ has a ground state solution if and only if $\alpha<\lambda_1$, where $\lambda_1$ is the first eigenvalue of $-\Delta$ on $\Omega$.  In order to apply Proposition~\ref{prop0001}, we will observe some bifurcation results of $u_\alpha$.  We believe that our observations are not new but since we could not find any convenient reference, we give their proofs below by Rabinowitz's global bifurcation theorem.
\begin{lemma}\label{lem1001}
There exists $\widetilde{\lambda}_0\in(0, \lambda_1)$ such that $u_\alpha$ is the unique ground state solution of $(\mathcal{P}_{\alpha})$ for $\alpha\in(\widetilde{\lambda}_0, \lambda_1)$.
\end{lemma}
\begin{proof}
Suppose $\alpha_n\uparrow\lambda_1$ as $n\to\infty$ and $\widetilde{u}_{\alpha_n}$ is a ground state solution of $(\mathcal{P}_{\alpha_n})$.  Similarly as in the proof of Lemma~\ref{lem0004}, we obtain that  $\widetilde{u}_{\alpha_n}\to 0$ strongly in $H_0^1(\Omega)$ as $n\to\infty$.  Thus, $\{(\widetilde{u}_\alpha, \alpha)\}$ are nontrivial branches of solutions to $(\mathcal{P}_{\alpha})$ bifurcated from the trivial branch of solutions $\{(0, \alpha)\}$ of $(\mathcal{P}_{\alpha})$ at $(0, \lambda_1)$.  In particular, $\{(u_\alpha, \alpha)\}$ is also a nontrivial branch of solutions to $(\mathcal{P}_{\alpha})$ bifurcated from the trivial branch  $\{(0, \alpha)\}$ of $(\mathcal{P}_{\alpha})$ at $(0, \lambda_1)$.  Note that $ \lambda_1$ is the principal eigenvalue of the linearized equation of $(\mathcal{P}_{\alpha})$ at $0$.  Hence, there is a unique continuous branch of solutions for $(\mathcal{P}_{\alpha})$ bifurcated from the trivial branch of solutions $\{(0, \alpha)\}$ of $(\mathcal{P}_{\alpha})$ at $(0, \lambda_1)$ near $\lambda_1$, say $\alpha\in(\widetilde{\lambda}_0, \lambda_1)$ for some $\widetilde{\lambda}_0>0$.  It follows that $\{(u_\alpha, \alpha)\}$ is the unique continuous branch of solutions for $(\mathcal{P}_{\alpha})$ bifurcated from the trivial branch of solutions $\{(0, \alpha)\}$ of $(\mathcal{P}_{\alpha})$ at $(0, \lambda_1)$ for $\alpha\in(\widetilde{\lambda}_0, \lambda_1)$.
\end{proof}

\begin{lemma}\label{lem1002}
The interval $(0, \lambda_1)$ is contained in the branch $\{(u_\alpha, \alpha)\}$.  Furthermore, for every $\alpha_n\downarrow0$ as $n\to\infty$, it holds that  $u_{\alpha_n}=u_0+o_n(1)$ strongly in $\h$, where $u_0$ is a ground state solution of $(\mathcal{P}_0)$.
\end{lemma}
\begin{proof}
Denote $\mathcal{L}=\{(u_\alpha, \alpha)\}$ and define $\lambda_0=\inf\{\alpha\mid(u_\alpha, \alpha)\in\mathcal{L}\}$.  Then it is easy to see that $-\infty\leq\lambda_0\leq\widetilde{\lambda}_0$.  We first prove that $\lambda_0\leq0$.  Suppose on the contrary, then by a similar argument as used in \cite[Lemma~5.2]{HWW14}, we can see that $m_{\alpha}<m_0$ for all $\alpha\in(\lambda_0, \lambda_1)$.  It follows that $\|u_\alpha\|_{L^2(\Omega)}^2<\bigg(\frac{2p}{p-2}m_0\bigg)^{\frac{2}{p}}|\Omega|^{\frac{p-2}{p}}$ for all $\alpha\in(\lambda_0, \lambda_1)$.  Let $\mathcal{U}=\{(u,\alpha)\mid\|u_\alpha\|_{L^2(\Omega)}^2<\bigg(\frac{2p}{p-2}m_0\bigg)^{\frac{2}{p}}|\Omega|^{\frac{p-2}{p}}\text{ and }\alpha\in\bbr\}$, then $\mathcal{U}$ is an open set in $L^2(\Omega)\times\bbr$, which contains the point $(0, \lambda_1)$.  By Lemma~\ref{lem1001}, $\mathcal{L}$ is a branch bifurcated from the point $(0, \lambda_1)$.  Since $\lambda_1$ is the principal eigenvalue of $-\Delta$ on $\Omega$ and there is no solution for $(\mathcal{P}_{\alpha})$ if $\alpha\geq\lambda_1$, we must get a contradiction due to Rabinowitz's global bifurcation theorem. Thus, the interval $(0, \lambda_1)$ is contained in the branch $\mathcal{L}$.  Now, by a similar argument in the proof of \cite[Lemma~5.2]{HWW14}, we can see that $u_{\alpha_n}=u_0+o_n(1)$ strongly in $H_0^1(\Omega)$ for every $\alpha_n\downarrow0$ as $n\to\infty$, where $u_0$ is a ground state solution of $(\mathcal{P}_0)$.
\end{proof}

By Lemma~\ref{lem1002}, we obtain a continuous curve in $\mathbb{S}:=\{(u_\alpha,\alpha)\mid u_\alpha\text{ is a solution of }(\mathcal{P}_\alpha)\}$, so that we can get the kinds of solutions described in Proposition~\ref{prop0001} by using similar arguments in the proof of Theorem~\ref{thm0001}.

\end{document}